\documentclass[a4paper]{article}

\usepackage{
amsmath,
amsthm,
amscd,
amssymb,
wasysym
}
\usepackage{mathtools}
\usepackage{comment}
\usepackage{tikz}
\usetikzlibrary{positioning,arrows,calc}
\usepackage{xspace}

\usepackage{graphicx}

\usepackage[
            left=3cm,right=3cm,top=3cm,bottom=3cm,%
            ]{geometry}
            
\makeatletter
\newcommand{\specialcell}[1]{\ifmeasuring@#1\else\omit$\displaystyle#1$\ignorespaces\fi}
\makeatother

\usepackage{url}

\theoremstyle{plain}

\newtheorem{theorem}{Theorem}

\newtheorem{question}{Question}

\newtheorem*{problem*}{Problem}
\newtheorem{claim}{Claim}

\newtheoremstyle{derp}
{3pt}
{3pt}
{}
{}
{\bfseries}
{:}
{.5em}
{}
\theoremstyle{derp}
\newtheorem{example}{Example}

\newcommand*{\claimproofname}{Proof of claim}
\newenvironment{claimproof}[1][\claimproofname]{\begin{proof}[#1]}{\end{proof}}

\newcommand{\Z}{\mathbb{Z}}

\newcommand{\N}{\mathbb{N}}

\newcommand\xqed[1]{%
  \leavevmode\unskip\penalty9999 \hbox{}\nobreak\hfill
  \quad\hbox{#1}}
\newcommand\qee{\xqed{$\fullmoon$}}

\newcommand{\End}{\mathrm{End}}
\newcommand{\Aut}{\mathrm{Aut}}
\newcommand{\Hom}{\mathrm{Hom}}

\title{Conjugacy of transitive SFTs minus periodic points}

\author{
Ville Salo \\
vosalo@utu.fi
}

\begin{document}
\maketitle


\begin{abstract}
It is a question of Hochman whether any two one-dimensional topologically mixing subshifts of finite type (SFTs) with the same entropy are topologically conjugate when their periodic points are removed. We give a negative answer, in fact we prove the stronger result that there is a canonical correspondence between topological conjugacies of transitive SFTs and topological conjugacies between the systems obtained by removing the periodic points. 
\end{abstract}

\section{Introduction}

Entropy (in its various forms) is arguably the most important invariant of isomorphism in dynamics, and a great deal of effort has been put into the question of how good an invariant it is. A classical theorem of this type is the result of Adler and Weiss from 1967 that two measure-preserving transformations coming from automorphisms of a $2$-torus are isomorphic if and only if their entropies coincide \cite{AdWe67}. Famously, Ornstein proved in 1970 \cite{Or70} that entropy is a complete invariant for the isomorphism of Bernoulli shifts. The result of Ornstein has been extended to (and in fact directly covers), a large class of systems. In particular, Friedman and Ornstein generalized this result in \cite{FrOr70} to mixing Markov shifts (more generally to ``weak Bernoulli shifts'').

In the topological setting, the simplest imaginable analog of Ornstein's theorem is trivially true: two full shifts $A^\Z, B^\Z$ are topologically conjugate if and only if their topological entropies agree ($\log |A| = \log |B| \iff |A| = |B|$). However, the natural analog of the Friedman-Ornstein result fails: it is easy to find two topologically mixing SFTs with the same topological entropy which are not isomorphic in the sense of topological conjugacy (an example is given below). An interesting result of Adler and Marcus from 1979 states that topological entropy is a complete invariant for a weaker kind of isomorphism called almost topological conjugacy \cite{AdMa79,LiMa95}, see also \cite{BoBuGo06} for a generalization.

One way to make entropy a better invariant in the topological category is to remove a set of ``bad points'' from the subshift. The celebrated result of Keane and Smodinsky from 1979 \cite{KeSm79}, a finitary strengthening of the result of Ornstein \cite{Or70}, can be interpreted in these terms: it proves that two Bernoulli shifts of equal-entropy are topologically conjugate outside a set of measure zero (and the conjugacy additionally respects the measures). More precisely, for equal entropy Bernoulli shifts $X, Y$, there exist measure zero sets $\bar X, \bar Y$ such that $X \setminus \bar X$ and $Y \setminus \bar Y$ are topologically conjugate. Adler and Marcus' almost topological conjugacy in \cite{AdMa79} implies the analogous result for mixing SFTs $X, Y$ with the same topological entropy (where the measure used is the one of maximal entropy).

In \cite{Ho13}, Hochman studied this question for a large class of systems including (topological) countable-state mixing Markov shifts (which generalize mixing SFTs), and showed that when we pass to the \emph{free part}, i.e.\ remove the periodic points, entropy becomes a complete invariant for Borel isomorphism off null sets, i.e.\ Borel isomorphism after removing a set having measure zero for \emph{all} invariant measures. In \cite{BoBuGo14}, Boyle, Buzzi and G\'{o}mez showed that for the slightly more restricted class of SPR Markov shifts (still including all mixing SFTs), the removal of a null set is not necessary. In particular, the free parts of mixing SFTs with equal entropy are Borel isomorphic. Hochman generalizes this result in \cite{Ho19}, covering for example mixing sofic shifts.

In \cite[Problem~1.9]{Ho13}, Hochman posed the following problem, which asks if, even for topological conjugacy, it suffices to remove the periodic points to make entropy a complete invariant.

\begin{problem*}
Let $X$, $Y$ be topologically mixing SFTs on finite alphabets, and $h(X) = h(Y)$. Let $X'$, $Y'$ denote the sets obtained by removing all periodic points from $X$, $Y$ respectively. Is there a topological conjugacy between the (non-compact) systems $X'$ and $Y'$?
\end{problem*}

This can be interpreted as asking for a common generalization of the results of \cite{AdMa79} (homeomorphism after removing a measure zero set) and \cite{BoBuGo14} (Borel isomorphism of $X'$ and $Y'$). The question is also quoted in \cite{BoBuGo14,BoBuMc15,Ho19}.




After removing the periodic points, the topological spaces $X', Y'$ are no longer compact, or even locally compact (though they are $G_\delta$, thus Polish spaces). This makes their topological dynamics tricky, and a practical consequence for their symbolic dynamics is the failure of the Curtis-Hedlund-Lyndon theorem: shift-commuting continuous functions between subshifts minus their periodic points $X'$, $Y'$ are not necessarily defined by local rules, at least not by ones with finite radius (this already happens in the case where $X, Y$ are mixing SFTs, see Example~\ref{ex:NotForContinuousMaps}).

No computable invariant for topological conjugacy of topologically mixing SFTs is known. Since after removing the periodic points we have more potential morphisms to consider, one might expect that it is even more difficult to show that two systems are non-isomorphic. Indeed, before this paper it was open even whether the binary full shift (the vertex shift defined by the matrix $\begin{psmallmatrix} 1 & 1 \\ 1 & 1 \end{psmallmatrix}$) becomes isomorphic to the subshift of proper $3$-colorings of the Cayley graph of $\Z$ (i.e.\ the vertex shift defined by the matrix $\begin{psmallmatrix} 0 & 1 & 1 \\ 1 & 0 & 1 \\ 1 & 1 & 0 \end{psmallmatrix}$) when periodic points are removed. Note that these SFTs trivially have the same entropy $\log 2$, as for each of them the number of legal paths of length $n$ in the corresponding graph is $\Theta(2^n)$. This special case of the problem was posed by Hochman in the 2010 Pingree Park Dynamical Systems Workshop.

We prove the following theorem.

\begin{theorem}
\label{thm:Main}
Let $X, Y$ be infinite topologically transitive SFTs. Every topological conjugacy $\phi' : X' \to Y'$ is a domain-codomain restriction of a topological conjugacy $\phi : X \to Y$, and the choice of this $\phi$ is unique.
\end{theorem}

In particular, the classification of topologically transitive SFTs stays the same (i.e.\ it is classified by strong shift equivalence) even if one removes the periodic points from all of them. Intuitively, this theorem shows that the aperiodic points of an SFT ``remember'' the periodic ones. It is well known that two mixing SFTs can have the same entropy without being topologically conjugate -- a standard example is the pair of vertex shifts mentioned above. Thus our theorem answers Hochman's problem in the negative.

We initially proved the following theorem, to answer a technical question about $\Aut(X')$ by Nishant Chandgotia (who was interested in this group due to Hochman's problem). It is an easy corollary of Theorem~\ref{thm:Main}.

\begin{theorem}
\label{thm:AutIso}
Let $X$ be an infinite topologically transitive SFT. Then $\Aut(X)$ is isomorphic to $\Aut(X')$.
\end{theorem}

As in Hochman's paper \cite{Ho13}, our main interpretation of these statements is in the two-sided case, and with this interpretation we prove the theorems under slightly weaker assumptions in Section~\ref{sec:TwoSided}. In Section~\ref{sec:OneSided} we prove a one-sided variant (where the shift action is only partial). In this setting, we obtain a stronger statement that applies to all shift-commuting continuous maps (and in the one-sided variant of Theorem~\ref{thm:AutIso} one therefore obtains $\End(X) \cong \End(X')$). 

In Section~\ref{sec:Examples} we show that the statement of Theorem~\ref{thm:Main} does not extend to topological conjugacies between two-sided mixing sofic shifts, or subshifts where the eventually periodic points are removed (for various meanings of the term). It also does not extend to general shift-commuting continuous maps between two-sided mixing SFTs, meaning there are indeed more morphisms in the category of free parts of mixing SFTs than there are in that of mixing SFTs, even though by our theorem the isomorphisms are exactly the same. The examples also show that for mixing sofic shifts $X, Y$, we \emph{can} have $X \not\cong Y$ yet $X' \cong Y'$.

Open problems are stated in Section~\ref{sec:Open}.

\section{Definitions}

We have $0 \in \N$, ${\subset}$ means the same as $\subseteq$ and $[a,b] = \{c \;|\; c \in \Z, a \leq c \leq b\}$. Let $A$ denote a finite set, called an \emph{alphabet}. The elements of $A$ are called \emph{symbols}. Write $A^* = \bigcup_{n \in \N} A^n$ for the set of all finite words over the alphabet $A$. These can be seen as the elements of the free monoid on the generating set $A$, and we write $uv$ for the product of words $u, v \in A^*$, i.e.\ their concatenation. One-sided (left- or right-) infinite words $u \in A^\N \cup A^{-\N}$ can also be concatenated with finite words, and the meaning should be clear. We index words (and one- or two-sided infinite words) with subscripts. 

See \cite{LiMa95} for a basic reference on symbolic dynamics. The set $A^\Z$ of two-way infinite words over a finite alphabet $A$ is called the \emph{full shift} (on alphabet $A$). Words that belong to a subshift are also called \emph{points}. The \emph{shift} $\sigma : A^\Z \to A^\Z$ defined by $\sigma(x)_i = x_{i+1}$ makes it a compact dynamical system. If $x \in A^\Z$, we write $u \sqsubset x$ if $x_{[i,i+|u|-1]} = u$ for some $i \in \Z$. If $X \subset A^\Z$ write $u \sqsubset X$ if $\exists x \in X: u \sqsubset x$. A \emph{subshift} is a subset $X$ of $A^\Z$ defined by forbidding a set of subwords $F \subset A^*$, in the sense that $X = \{x \in A^\Z \;|\; \forall u \in F: u \not\sqsubset x\}$. Subshifts are exactly the closed $\sigma$-invariant sets in $A^\Z$. An \emph{SFT} is a subshift where the defining set of forbidden words can be taken to be finite, and a \emph{sofic shift} is an image of an SFT under a shift-commuting continuous function between two full shifts.  The \emph{vertex shift} defined by the $n$-by-$n$ matrix $M$ over natural numbers is the SFT $\{x \in \{1,\ldots,n\}^\Z \;|\; \forall i: M_{x_i,x_{i+1}} = 1\}$; every SFT is topologically conjugate to a vertex shift.

Subshifts form a category, with \emph{morphisms} the shift-commuting continuous functions. The \emph{automorphism group} of a subshift $X$, denoted $\Aut(X)$, is the group of its shift-commuting self-homeomorphisms under function composition. More generally, the \emph{endomorphism monoid} $\End(X)$ consists of the shift-commuting continuous self-maps. A \emph{topological conjugacy} between two subshifts is a shift-commuting homeomorphism between them. Of course, when the domain is compact and the image is Hausdorff, it suffices to find a shift-commuting continuous bijection. If a property of dynamical systems is invariant under topological conjugacy, we say it is \emph{dynamical}.

More generally a (non-compact) \emph{dynamical system} is a topological space with a continuous $\Z$-action (or $\N$-action), and we use the same definitions of topological conjugacy, morphism, automorphism group and endomorphism monoid for such systems. Hochman's original problem deals with the systems $X' = \{x \in X \;|\; x \mbox{ is not periodic}\}$ where $X$ is a subshift. The restriction of $\sigma$ is of course well-defined and continuous on this set when we take the subspace topology for $X'$.

A subshift $X \subset A^\Z$ is \emph{transitive} if for all $u, v \sqsubset X$, there exists $w$ such that $uwv \sqsubset X$, and \emph{mixing} if for any $u,v$ such a word $w$ can be found of any large enough length. If $X \subset A^\Z$ is a subshift, then we can see the dynamical system $(X,\sigma^m)$ as a subshift by blocking consecutive $A$-words of length $m$ into a new product alphabet $A^m$ to get a subshift in $(A^m)^\Z$. The \emph{language} of a subshift $X \subset A^\Z$ is the set $\{u \in A^* \;|\; u \sqsubset X\}$. The \emph{entropy} of subshift $X$ is $\lim_n \frac{\log P_n(X)}{n}$ where $P_n(X)$ denotes the number of words of length $n$ in the language of $X$.

A useful construction for us is ``passing to a power of a shift''. Namely, if $X$ is a subshift, then for any $m > 0$, $(X, \sigma^m)$ is a dynamical system, and in fact itself a subshift in an obvious sense \cite{LiMa95}.

A \emph{periodic point} in $X \subset A^\Z$ is $x \in X$ satisfying $\sigma^n(x) = x$ for some $n > 0$. We similarly define periodic points of $X \subset A^\N$ (using the same formula) and $X \subset A^{-\N}$ (shifting in the other direction). An \emph{aperiodic point} is a point that is not periodic. For $a \in A$ write $a^\Z$ for the unique $\sigma$-fixed point $x \in A^\Z$ satisfying $x_0 = a$. Define $a^{\N}$ and $a^{-\N}$ similarly.

If $X \subset A^\Z$ is a subshift, $X_{\N} = \{y \in A^\N \;|\; \exists x \in X: \forall i \in \N: y_i = x_i\}$ is a closed set which is invariant under the map $\sigma : A^\N \to A^\N$, $\sigma(X) = X$, where $\sigma$ is defined by the same formula as in the two-sided case, i.e.\ it is an $\N$-subshift. We similarly define $X_{-\N}$, and functions $x \mapsto x_{\N} : X \to X_{\N}$ and $x \mapsto x_{-\N} : X \to X_{-\N}$ in the obvious way. An \emph{isolated} point in a subshift $X_\N$ is one that is topologically isolated. An important property of a subshift $X$ (in this paper) is that $X_{\N}$ has dense aperiodic points. The reader may want to check that this is a dynamical property of $X$.

If $X \subset A^\Z$ is a subshift, a word $w \sqsubset X$ is \emph{synchronizing} if whenever $x, y \in X$ and $x_{[\ell,\ell']} = w = y_{[\ell,\ell']}$, the point $z$ defined by
\[ z_i = \left\{ \begin{array}{ll}
x_i & \mbox{ if } i < \ell, \\
w_i & \mbox{ if } \ell \leq i \leq \ell', \\
y_i & \mbox{ if } i > \ell' \\
\end{array} \right. \]
is also in $X$; a point is synchronizing if it contains a synchronizing word. An important property of a subshift $X$ (in this paper) is that all its periodic points are synchronizing. The reader may want to check that this is a dynamical property of $X$.

\section{Two-sided subshifts}
\label{sec:TwoSided}

\begin{theorem}
\label{thm:Technical}
Let $X$ and $Y$ be two-sided subshifts. Suppose every periodic point $s \in X$ is synchronizing, $X_\N$ and $X_{-\N}$ have dense aperiodic points, and the same assumptions hold for $Y$. Then every topological conjugacy $\phi' : X' \to Y'$ is a domain-codomain restriction of a topological conjugacy $\phi : X \to Y$, and the choice of this $\phi$ is unique.
\end{theorem}

\begin{proof}
Suppose $X, Y \subset A^\Z$ for some finite alphabet $A$. The non-trivial direction is to show that a topological conjugacy $\phi' : X' \to Y'$ extends uniquely to a conjugacy $\phi : X \to Y$. Consider
\[ Z = \{(x, \phi'(x)) \;|\; x \in X\} \subset X' \times Y' \]
Let $R$ be the closure $\bar Z \subset X \times Y$, and let us think of it as a relation. Note that the system $\bar Z$, under the diagonal action $\sigma(x, y) = (\sigma(x), \sigma(y))$, can be seen naturally as a subshift over the alphabet $A^2$ (the closure of a shift-invariant set is shift-invariant).

We use relation notation, and write $x R = \{y \in Y \;|\; (x,y) \in \bar Z\}$ and $R y = \{x \in X \;|\; (x,y) \in \bar Z\}$, and similarly for subsets in place of $x$ and $y$. Note that we have $X R = Y$ and $R Y = X$ because aperiodic points are dense in $X$ and $Y$ (which follows from the corresponding assumptions on $X_\N$ and $Y_\N$). Our first claim is an immediate consequence of the fact $\phi'$ is a homeomorphism.

\begin{claim}
\label{clm:AperiodicSingleton}
The sets $xR$ and $Ry$ are singletons for any $x \in X'$ and $y \in Y'$.
\end{claim}

We now show that proving the same for periodic points suffices to prove the theorem.

\begin{claim}
\label{clm:IfSingleton}
If $xR$ and $Ry$ are singletons for every periodic point $x \in X$ and every periodic point $y \in Y$, then the result holds.
\end{claim}

\begin{claimproof}
Extend $\phi'$ to $X$ by mapping each periodic point $x$ to the unique element in $xR$, 
and extend $(\phi')^{-1}$ by mapping $y$ to the unique point in $Ry$. The two extensions are clearly inverses. It remains to show that the extension $\phi$ is shift-commuting and continuous. Continuity follows from the closed graph theorem, since by definition the graph $R$ of $\phi$ is closed, and $Y$ is a compact Hausdorff space. Shift-commutation is assumed on aperiodic points, and by density of aperiodic points and continuity it then also holds on periodic points. Finally, it is clear that the choice of the extension $\phi$ is unique, again because aperiodic points are dense.
\end{claimproof}

We now prove that $xR$ and $Ry$ are indeed singletons for any periodic points $x$, $y$, under the assumptions of the theorem. Since we prove this for all systems at once and the roles of $X$ and $Y$ are symmetric, it suffices to prove this for sets of the form $xR$. Furthermore, since $(X, \sigma^m)$ satisfies the dynamical assumptions whenever $(X, \sigma)$ does, it is enough to prove it for fixed points for the dynamics, by replacing $\sigma$ by $\sigma^m$ (for both $X$ and $Y$).

\begin{claim}
If $s \in X$ is a fixed point for $\sigma$, then $sR$ is a finite subshift.
\end{claim}

\begin{claimproof}
It is clear that $W = (\{s\} \times Y) \cap \bar Z$ is a subshift (under the diagonal action) as the intersection of two subshifts, and $sR$ is just the projection of $W$ to $Y$, thus it is a subshift. Suppose for a contradiction that this subshift is infinite. Every infinite subshift contains an aperiodic point, see \cite[Theorem~3.8]{BaDuJe08} (and the sentence after it). Thus, there is an aperiodic point $y \in sR$. It follows that $Ry$ contains both $s$ and the aperiodic point $(\phi')^{-1}(y)$, contradicting Claim~\ref{clm:AperiodicSingleton}.
\end{claimproof}

Now observe that if $sR$ is a finite subshift, there exists $m > 0$ such that $\sigma^m(y) = y$ for all $y \in sR$. If we further replace $\sigma$ by $\sigma^m$ (again both for $X$ and $Y$), $s$ will stay a fixed point, and $sR$ becomes pointwise stabilized by the dynamics. It is thus enough to prove that $sR$ is a singleton under this assumption.

The next crucial observation is that if $s = a^\Z$ and $sR$ is fixed pointwise by $\sigma$, then every long $a^*$-segment of an aperiodic point must be mapped under $\phi'$ to a long $b^*$-segment, for some symbol $b \in A$, and the length difference is bounded. This is a compactness argument.\footnote{Dynamically, we are using the fact that every finite dynamical system has the pseudo-orbit tracing property.} 

\begin{claim}
Suppose $a \in A$, $s = a^\Z$ and $\sigma(y) = y$ for all $y \in sR$. Then there exists $n$ such that for any $x \in X'$, if $x_{[i-n,i+1+n]} = a^{2n+2}$ for $i \in \Z$, then $\phi'(x)_i = \phi'(x)_{i+1}$.
\end{claim}

\begin{claimproof}
Suppose not. Then shifting such points $x$ by $\sigma^i$, for each $n$ we can find an aperiodic point $x^n$ such that $x^n_{[-n,1+n]} = a^{2n+2}$, and $\phi'(x)_0 \neq \phi'(x)_1$. But $x^n \rightarrow s$ as $n \rightarrow \infty$, and clearly no limit point of the sequence $\phi'(x^n)$ is fixed by $\sigma$, contradicting the assumption that $sR$ contains only fixed points.
\end{claimproof}

It is useful to once again replace $\sigma$ by a power: replacing it with $\sigma^n$ changes $n$ to $1$ in the previous claim: the symbol $a$ is replaced by the symbol $a^n \in A^n$, and if we see $a^n a^n a^n a^n$ in a configuration, then by applying the previous claim in every position, we see that the corresponding word in the image must be some $ub^{2n}v$ where $|u| = |v| = n$, i.e.\ the symbols are $u, b^n, b^n, v$, of which the central two are equal.

At this point, we may therefore assume that whenever $x_{[i,i+3]} = aaaa$, we have $\phi'(x)_{i+1} = \phi'(x)_{i+2}$. So far, we have not used any property of the subshift $X$ except the density of aperiodic points. The proof of the next claim crucially depends on the two dynamical properties assumed.

\begin{claim}
\label{clm:IsSingleton}
Suppose $s \in X$ is periodic. Then $sR$ is a singleton.
\end{claim}

\begin{claimproof}
As we have deduced, we may suppose $s = a^\Z$ for $a \in A$, $\sigma(y) = y$ for all $y \in sR$, and that whenever $x_{[i,i+3]} = aaaa$, we have $\phi'(x)_{i+1} = \phi'(x)_{i+2}$.

Consider the set $W$ of all points $x \in X'$ such that $x_i = a$ for all $i \in \N$ and $x_{-1} \neq a$. Note that $W$ is nonempty, as otherwise the point $a^{-\N}$ is isolated in $X_{-\N}$, which contradicts the assumption that aperiodic points are dense in $X_{-\N}$.

Observe that every point $z \in \phi'(W)$ satisfies $z_i = z_{i+1}$ for all $i \geq 1$. Let $B \subset A$ be the set of all possible symbols $z_1$ that appear. The set $W$ is compact so $\phi' : X' \to Y'$ is uniformly continuous in $W$.\footnote{To spell out a topological ``subtlety'' here, we mean more than just $\phi'|_W : W \to Y'$ being uniformly continuous, namely $\forall \epsilon > 0 : \exists \delta > 0: \forall w \in W, x \in X': d(w, x) < \delta \implies d(\phi'(w), \phi'(x)) < \epsilon$, where $d$ is the induced metric from any metric on $X$.} This means that there exists $m$ such that whenever $x \in X'$ and $x_{[-m,m]} = w_{[-m,m]}$ for some $w \in W$, we have $\phi'(x)_1 = \phi'(w)_1$. Note that the assumption ``$\exists w \in W: x_{[-m,m]} = w_{[-m,m]}$'' just means $x_{[0,m]} = a^{m+1}$, $x_{-1} \neq a$.

Replacing $\sigma$ with $\sigma^m$, we may assume that for all $x \in X'$, if $x_{[-1,2]} = a'aaa$ and $a' \neq a$ then $\phi'(x)_1 \in B$ is determined by $a'$. More precisely, there exists a function $F : A \setminus \{a\} \to B$ such that
\[ \forall x \in X': x_{[-1,2]} = a'aaa \wedge a' \neq a \implies \phi'(x)_1 = F(a'). \]
Now perform a left-right symmetric argument (and possibly replace $\sigma$ by a power yet again) to obtain that whenever $x_{[-2,1]} = aaaa'$ where $a' \neq a$, the symbol $\phi'(x)_{-1}$ is determined by the choice of $a'$. Let $C \subset A$ be, symmetrically with how $B$ was defined, the symbols that appear as such $\phi'(x)_{-1}$.

Applying the conclusions of the two paragraphs above on both sides of a finite interval (the latter through a shift), and also applying the property from the first paragraph (which does not disappear when passing to a power of the shift, as explained above the present claim), we deduce that whenever $x_{[j,j+k]} = a^{k+1}$ with $k \geq 2$ and $a \notin \{x_{j-1}, x_{j+k+1}\}$, we have $\phi'(x)_i = \phi'(x)_{i+1}$ for all $i \in [j+1,j+k-2]$, and the symbol $\phi'(x)_{j+1} = \phi'(x)_{j+k-1}$ is determined uniquely by both $x_{[j-1,j+2]}$ and by $x_{[j+k-2,j+k+1]}$.

Now we use the fact that $a^\Z$ is synchronizing. Of course, we have passed to a power of the shift, but the reader can easily verify that a synchronizing point stays synchronizing when we pass to a power of the shift. We obtain $\ell$ such that if $x,y \in X$ satisfy $x_{[0,\ell-1]} = y_{[0,\ell-1]} = a^{\ell}$, then the point $z$ defined by
\[ z_i = \left\{ \begin{array}{ll}
x_i & \mbox{ if } i < 0, \\
a & \mbox{ if } 0 \leq i < \ell, \\
y_i & \mbox{ if } i \geq \ell \\
\end{array} \right. \]
is also in $X$. 
By replacing $\sigma$ with $\sigma^{\ell}$, we may assume $\ell = 1$. 

Suppose that either $B$ or $C$ is not a singleton. Then in particular we can find $b \in B$ and $c \in C$ with $b \neq c$ such that for some $x, x' \in X$ we have $x_i = x'_{2-i} = a$ for all $i \geq 0$ and $x_{-1} \neq a$, $x'_3 \neq a$, and $\phi'(x)_1 = b \neq c = \phi'(x')_1$. Now apply the synchronization assumption with $\ell = 1$ to $x$ and $x'$, to obtain a point $x''$ with $x''_{[-1,2]} = x_{[-1,2]}$ and $x''_{[0,3]} = x'_{[0,3]}$. Suppose first that $x''$ is aperiodic. Then we have
\[ \phi'(x'')_1 = \phi'(x)_1 = b \neq c = \phi'(x')_1 = \phi'(x'')_1, \]
a contradiction.

If $x''$ is periodic, then we apply synchronization to $x$ and $\sigma^{-1}(x')$ (again with $\ell = 1$) to obtain a point $x'''$. The period of $x''$ cannot be $1$, because it contains the letter $a$ but is not equal to $a^\Z$, and the point $x'''$ is obtained from $x''$ by duplicating a single letter. Duplicating a single letter in a periodic point of period greater than $1$ always turns it aperiodic, so $x'''$ is aperiodic. Now a calculation analogous to that of the previous paragraph gives a contradiction.

We now have $B = C = \{b\}$ for some $b \in A$. This clearly implies $sR = \{b^\Z\}$, since in any aperiodic point close to $s$ the sequence of $a$s containing the origin eventually breaks either on the left or on the right, and a run of $b$s reaching the origin is forced by the assumptions.
\end{claimproof}

By Claim~\ref{clm:IfSingleton}, Claim~\ref{clm:IsSingleton} concludes the proof of the theorem.
\end{proof}

Theorem~\ref{thm:Main} follows directly, because any long enough word is synchronizing in an SFT, and one-sided infinite transitive SFTs do not have isolated points. For Theorem~\ref{thm:AutIso}, map $f \in \Aut(X)$ to $f|_{X',X'} : X' \to X'$. This is clearly well-defined and a group homomorphism. It is surjective by the first claim in Theorem~\ref{thm:Main}, and injective by the second, therefore it is a group isomorphism.

\section{One-sided subshifts}
\label{sec:OneSided}

Subshifts are two-sided in the paper of Hochman, so our Theorem~\ref{thm:Main} indeed solves his problem. However, as written, one can also interpret the problem in the one-sided category.

Directly removing the periodic points typically leads to a set which is not invariant for the shift action, so $\sigma$ is only a partial action. We again set $X' = \{x \in X \;|\; x \mbox{ is aperiodic}\}$. In this case, we take topological conjugacy to mean a homeomorphism that commutes with $\sigma$ whenever $\sigma$ is defined on $X'$, i.e.\ if both $x$ and $\sigma(x)$ are aperiodic, then $\sigma(\phi'(x)) = \phi'(\sigma(x))$ (in particular $\sigma(\phi'(x))$ must be aperiodic when the codomain of $\phi'$ is another subshift with periodic points removed).

Whether or not this one-sided interpretation is natural can be debated. Nevertheless, with this interpretation, we solve the problem of Hochman in a strong form, with a short proof.

\begin{theorem}
\label{thm:OneSided}
Suppose $Z \subset A^\Z$ is a subshift where every periodic point $s \in Z$ is synchronizing and neither $Z_{-\N}$ nor $Z_{\N}$ has an isolated periodic point, and let $X = Z_\N$. Then every shift-commuting continuous function $\phi' : X' \to Y'$ is a domain-codomain restriction of a shift-commuting continuous function $\phi : X \to Y$, and the choice of this $\phi$ is unique.
\end{theorem}
 
Recall that our one-sided subshifts are by definition restrictions of two-sided subshifts (we require $\sigma(X) = X$); alternatively one can require $\sigma(X) \subset X$. In either case, a one-sided \emph{transitive} subshift $X$ (transitivity being defined with the same formula as for two-sided SFTs) with forbidden patterns $F$ is equal to $Z_\N$ where $Z$ is the two-sided subshift with forbidden patterns $F$. Thus, infinite transitive one-dimensional SFTs satisfy the assumptions of the theorem under either definition.

\begin{proof}
Define the relation $R \subset X \times Y$ analogously to the two-sided case. Again it is enough to show that for every periodic point $x$, $xR$ is a singleton (whose unique element is then automatically a periodic point). Suppose that this is not the case. Again passing to a power of the shift we may assume that for some $a \in A$ and some $i \in \N$, for arbitrarily large $n$ we can find $x, x' \in X'$ such that $x_{[0,n-1]} = x'_{[0,n-1]} = a^n$, $x_n \neq x'_n$, and $\phi'(x)_i \neq \phi'(x')_i$.\footnote{The fact $i$ can be taken to be $0$ does follow from shift-commutation, but this needs a small additional argument since $x,x'$ might be eventually periodic, and it does not simplify the rest of the proof.}

Consider such a fixed $n$, and observe that by uniform continuity of $\phi'$ on $\{x,x'\}$ we can find $m \in \N$ and words $u, v \in A^m$ such that $u_0 \neq v_0$, for any two points $y,y' \in X'$ such that $y_{[0,n+m-1]} = a^n u, y'_{[0,n+m-1]} = a^n v$ we have $\phi'(y)_i \neq \phi'(y')_i$, and at least one such pair $y, y'$ exists (i.e.\ $a^n u$ and $a^n v$ appear in the language of $X$).

By the assumption that $Z_{-\N}$ has no isolated periodic points, there exists a point of the form $ba^\N$ in $X$, for some $b \neq a$. By synchronization of $a^\Z$, assuming $n$ was large enough we have $ba^{\ell}uz, ba^{\ell}vz' \in X'$, for some $z, z' \in X$ and any $\ell \geq n$. By repeatedly applying the assumption that $Z_\N$ has no isolated periodic points, and synchronization of periodic points, we may assume $z, z' \in X'$, implying that also $a^{\ell}uz, a^{\ell}vz' \in X'$.

The same $b \neq a$ appears for infinitely many $n$, and from
\[ \phi'(ba^{\ell}uz)_{i+1} = \phi'(a^{\ell}uz)_i \neq \phi'(a^{\ell}vz')_i = \phi'(ba^{\ell}vz')_{i+1} \]
it follows that $\phi'$ is not continuous at $ba^\N \in X'$.
\end{proof}

\section{Counterexamples to stronger statements}
\label{sec:Examples}

In this section, we give some counterexamples to strengthenings of theorems proved in the above sections, and of Hochman's original problem.

These are all sofic shifts, and we use regular expressions to define some of them; in these expressions, $+$ denotes union, $^*$ denotes repetition (zero or more times), and juxtaposition is concatenation. The reader can pick up the precise definition in any reference on formal language theory. The soficity of the subshifts below is easy to see, but we remark that it is a general fact that a subshift $X$ is sofic if and only if for some regular language $L$, it is the unique smallest subshift whose language contains $L$ (and such a unique smallest subshift exists).

Theorem~\ref{thm:Technical} does not cover mixing sofic shifts, and indeed the conclusion fails there:

\begin{example}
\label{ex:First}
For a two-sided mixing sofic shift $X$, an automorphism $f' \in \Aut(X')$ may not be the restriction of an automorphism $f \in \Aut(X)$: Consider the smallest subshift whose language contains $((0^*+1^*) 2 (0^*+1^*) 3)^*$. Define $f'$ by swapping each maximal  occurrence of $2a^n$ with the word $2(1-a)^n$ and each (possibly overlapping)  maximal occurrence of $a^n3$ with the word $(1-a)^n3$, for $a \in \{0,1\}$.
\qee
\end{example}

Note that the proof of Theorem~\ref{thm:Main} shows in general that the ``image subshift'' $sR$ must be finite for a periodic point $s$. The above example shows it can indeed have positive cardinality, but in the example the period is the same. We give another simple example where the period is different. It also shows that Hochman's original problem has a negative answer for mixing sofic shifts.

\begin{example}
Consider the two-sided \emph{even shift} $X \subset \{0,1\}^\Z$, namely the smallest subshift whose language contains $(1(00)^*)^*$, and $Y \subset \{0,1,2\}^\Z$ the smallest subshift whose language contains $(2(01)^*)^*$. These are mixing sofic shifts. Define $f' : X \to Y$ by rewriting each subword $0^{2n}1$ to $(01)^n2$, and $10^{2n}$ to $2(01)^n$. Clearly this gives a topological conjugacy between $X'$ and $Y'$, but $f'$ is not the restriction of a topological conjugacy between $X$ and $Y$. Indeed, $X$ and $Y$ are not topologically conjugate at all, because $X$ has two points of period $1$, while $Y$ has only one. \qee
\end{example}

In fact, it is a classical fact that even a mixing SFT can be conjugate to a mixing sofic shift after the periodic points are removed (we thank Mike Boyle for pointing out this example).

\begin{example}
A \emph{near Markov} \cite{BoKr88} subshift $Y$ is a sofic shift such that in the canonical irreducible SFT cover $X$, in the canonical projection $\pi : X \to Y$ only finitely many points have multiple preimages. One example is the even shift $X$ from the previous example: its canonical cover is the \emph{golden mean shift} (the mixing SFT over alphabet $\{0,1\}$ with unique forbidden pattern $11$) with the projection $\pi(x)_i \equiv 1 - (x_i + x_{i+1}) \bmod 2$. Only the point $1^\Z$ has multiple preimages. In this case, clearly the points with multiple preimages are all periodic, in particular the canonical projection $\pi$ restricted to $X'$ is a topological conjugacy between $X'$ and $Y'$. \qee
\end{example}

In fact, also the subshift $X$ in Example~\ref{ex:First} is near Markov. Thus, while in that example we showed that $\Aut(X')$ contains elements $\Aut(X)$ does not, we have canonical isomorphisms $\Aut(X') \cong \Aut(Z) \cong \Aut(Z)$, where $Z$ is the canonical cover of $X$.

The one-sided result in Theorem~\ref{thm:OneSided} does not extend to the two-sided case. In fact in the two-sided case, the subshift $sR$ need not even be finite when invertibility of $\phi'$ is not assumed.

\begin{example}
\label{ex:NotForContinuousMaps}
There are two-sided mixing SFTs $X, Y$ and a shift-commuting continuous map $f' : X' \to Y'$ which is not the restriction of any continuous shift-commuting map $f : X \to Y$. From the full shift on symbols $\{0, 2\}$, map to the full shift on symbols $\{0,1,2\}$ by mapping $2$ to $2$, and between two $2$s, in a segment $2 0^n 2$ write $1$ at the midpoint between the $2$s (with an arbitrary tie-breaking rule). Tails of the form $20^\N$ and $0^{-\N}2$ are fixed. This clearly defines a shift-commuting continuous map $f' : X' \to Y'$, but no extension of $f'$ to $X$ is continuous at $0^\Z$. \qee
\end{example}

By modifying this construction, we obtain that there are uncountably many shift-commuting continuous maps from $X'$ to $Y'$, so the set of morphisms $\Hom(X', Y')$ is significantly larger than $\Hom(X, Y)$. A further modification shows that $\End(X')$ is uncountable, in particular the monoids $\End(X)$ and $\End(X')$ are not isomorphic.

One may also ask what happens if we remove a larger set than just the periodic points. In the one-sided case $X \subset A^\N$, it is tempting to remove all \emph{eventually periodic points}, i.e.\ points $x \in X$ whose orbit is finite, as the resulting set $X''$ is the largest shift-invariant set not containing any periodic points. We give a stronger example.

\begin{example}
Let $X = \{0,1,2\}^\N$ and let
\[ Y = X \setminus \{x \in X \;|\; \sum_i x_i < \infty\}. \]
Let $g \in \Aut(Y)$ be map that rewrites $a \in \{1,2\}$ to $3-a$ if the distance to the nearest symbol from $\{1,2\}$ on the right is odd, and otherwise does not. On $Y$, this map is continuous, because for all $x \in Y$ we can find an open neighborhood that specifies the positions of at least $n$ symbols from the alphabet $\{1,2\}$, and this determines the new values at at least the first $n-1$ positions. Clearly this map has no continuous extension to any point $x \in X \setminus Y$.

It is easy to see that the restriction of $g$ to $X''$ gives an element of $\Aut(X'')$, where $X''$ is $X$ without its eventually periodic points (note that $g$ is an involution and its definition clearly implies that the image of an eventually periodic point is eventually periodic), and there is clearly no continuous extension to $10^\N$ (for this we only need that points in $X''$ can begin with words $1 0^n a$ with $a \neq 0$, where $n$ is arbitrarily large and of arbitrary parity). \qee
\end{example}

In the two-sided case, one can remove points that have an eventually periodic right tail or left tail (or one can remove both types of points, or points that are of both types simultaneously), or one can remove points that agree with some periodic point in all but finitely many positions. The following example covers all these cases, by a similar argument as in the previous example.

\begin{example}
\label{ex:TwoSidedRemoval}
Let $X = \{0,1,2\}^\Z$ and let
\[ Y = X \setminus \{x \in X \;|\; \exists n \in \Z: x_n \in \{1,2\} \wedge \forall i \neq n: x_i = 0\} \]
(i.e.\ we remove only two orbits from $Y$). Let $g \in \Aut(Y)$ be map that rewrites $a \in \{1,2\}$ to $3-a$ if the distance to the nearest symbol from $\{1,2\}$ is odd, and otherwise does not. On $Y$, this map is continuous, because for all $x \in Y$ we can find an open neighborhood that specifies at least two non-zero coordinates, and this gives a bound on how far we have to look to deduce the $i$th coordinate of the image of $x$. Clearly this map has no continuous extension to any point $x \in X \setminus Y$. \qee 
\end{example}

\section{Open problems}
\label{sec:Open}


\begin{question}
How much do we need to remove from mixing SFTs $X, Y$ to make entropy a complete invariant for topological conjugacy? Is removing the eventually periodic points enough?
\end{question}

Slightly more precisely, let $\phi$ be a function that, given a subshift $X$, produces a subset of $X$ (with the intuition that some ``bad points'' are removed). For which $\phi$ do we have $h(X) = h(Y) \iff \phi(X) \cong \phi(Y)$ for mixing SFTs $X, Y$? Of course this is only interesting for ``natural'' choices of $\phi$, such as those listed above Example~\ref{ex:TwoSidedRemoval}.


We also note that Theorem~\ref{thm:Main} is truly about topological conjugacies, and says nothing about \emph{embeddings} (shift-commuting continuous injections) or \emph{factorings} (shift-commuting continuous surjections). There are also new topological subtleties, for example, at least as written the proof of Theorem~\ref{thm:Main} does not even deal with the case where $\phi : X' \to Y'$ is only a continuous shift-commuting bijection. To our knowledge, little is known about these issues for free parts of mixing SFTs with equal entropy. For example the following problems stay open.

\begin{question}
Let $X, Y$ be respectively the vertex shifts defined by $\begin{psmallmatrix} 1 & 1 \\ 1 & 1 \end{psmallmatrix}$ and $\begin{psmallmatrix} 0 & 1 & 1 \\ 1 & 0 & 1 \\ 1 & 1 & 0 \end{psmallmatrix}$. Is there an embedding relation between $X'$ and $Y'$? Does $X'$ factor onto $Y'$ (or even onto $Y$)?
\end{question}

It is clear that $Y$ factors onto $X$, so $Y'$ factors onto $X'$. However, $X$ does not factor onto $Y$, and neither embeds in the other, due to periodic point restrictions. (There \emph{does} exist a continuous shift-invariant \emph{map} from $X'$ into $Y'$, even though there does not exist one from $X$ to $Y$.)



In higher-dimensional settings, we do not know to what extent variants of Theorem~\ref{thm:Main} hold; some parts of our argument have direct analogs, some do not seem to. In the two-dimensional case, it is known that the set of totally aperiodic points in the binary full shift does not even admit a continuous shift-commuting map to the space of proper $3$-colorings of the standard grid \cite{GaJaKrSe18}.

Once there exists an isomorphism between two systems, it is of interest to try to understand the family of all isomorphisms, which boils down to the study of the automorphism group. We showed in Example~\ref{ex:First} that automorphisms of free parts of mixing sofic shifts are no longer canonically isomorphic to those of the original sofic shift (i.e.\ there may be new automorphisms that are not restrictions of old ones). However, for near-Markov sofics $X$ (such that as that of Example~\ref{ex:First}), $\Aut(X')$ is canonically isomorphic to the automorphism group of a mixing SFT, namely the canonical cover. What can be said about $\Aut(X')$ for a general mixing sofic shift $X$? Can we find a concrete characterization of its elements?

\section*{Acknowledgements}

We thank Mike Boyle, Nishant Chandgotia, Mike Hochman, Johan Kopra and the anonymous referee for helpful comments and suggestions. The author was supported by Academy of Finland grant 2608073211.

\bibliographystyle{plain}
\bibliography{../../../bib/bib}{}

\end{document}